\documentclass[12pt,oneside]{amsart}

\usepackage[T1]{fontenc}
\usepackage[utf8]{inputenc}
\usepackage{lmodern}
\usepackage{microtype}
\usepackage{amsmath,amssymb,amsthm,mathtools}
\usepackage{enumitem}
\usepackage[colorlinks=true,linkcolor=blue,citecolor=blue,urlcolor=blue]{hyperref}
\newcommand{\doi}[1]{\href{https://doi.org/#1}{doi:\,#1}}

\newtheorem{theorem}{Theorem}[section]

\newtheorem{lemma}[theorem]{Lemma}
\newtheorem{corollary}[theorem]{Corollary}

\newtheorem{remark}[theorem]{Remark}
\newtheorem*{theoremA}{Theorem A}
\newtheorem*{theoremB}{Theorem B}

\newcommand{\R}{\mathbb R}

\newcommand{\N}{\mathbb N}
\newcommand{\cL}{\mathcal L}
\newcommand{\cM}{\mathcal M}
\newcommand{\cN}{\mathcal N}
\newcommand{\cR}{\mathcal R}
\newcommand{\Sym}{\operatorname{Sym}}
\newcommand{\dist}{\operatorname{dist}}

\newcommand{\Tr}{\operatorname{Tr}}

\newcommand{\HS}{\mathfrak S_2}
\newcommand{\dual}[2]{\langle #1,#2\rangle}

\title[Stability from Uniqueness for Inverse Problems]{H\"older Stability from Exact Uniqueness for Finite-Dimensional Analytic Inverse Problems}
\author[C. I. C\^{a}rstea]{C\u{a}t\u{a}lin I. C\^{a}rstea}
\address{Department of Applied Mathematics, National Yang Ming Chiao Tung University, Hsinchu, Taiwan}
\email{catalin.carstea@gmail.com}
\date{May 2026}

\begin{document}

\begin{abstract}
We prove a stability theorem for finite-dimensional analytic inverse problems.  Let
\(U\subset\R^m\) be an open parameter set, let \(F(p)\) be a boundary measurement
operator, and let \(R(p)\) be the finite-dimensional quantity to be recovered.  If
\(F\) is real analytic and
\[
        F(p)=F(q)\quad\Longrightarrow\quad R(p)=R(q),
\]
then \(R\) satisfies a H\"older stability estimate on every compact subset of \(U\).
The proof uses a Hilbert--Schmidt scalarization of the operator equation
\(F(p)=F(q)\) and the \L{}ojasiewicz distance inequality.  We also prove that,
after fixing countable dense families of boundary inputs and tests, finitely many
scalar matrix elements of the data give the same H\"older recovery on compact
parameter sets.  This finite-measurement conclusion is qualitative: the proof does
not give an effective measurement list, exponent, or constant.  The
finite-measurement statement follows from finite determinacy of real analytic
zero sets.  We apply the result to local Neumann-to-Dirichlet
data for piecewise constant anisotropic conductivities and to localized
Dirichlet-to-Neumann data for piecewise homogeneous anisotropic elasticity.
\end{abstract}

\maketitle

\section{Introduction}

Inverse boundary value problems ask whether one can determine an unknown
coefficient inside a domain from measurements made only at the boundary.  A
model example is the conductivity equation
\[
        \nabla\cdot(\gamma\nabla u)=0\quad \hbox{in }\Omega,
        \qquad u|_{\partial\Omega}=f,
\]
where the coefficient \(\gamma\) describes an interior material property.  The
Dirichlet-to-Neumann map sends the imposed boundary voltage \(f\) to the
corresponding boundary current
\[
        \Lambda_\gamma f = \gamma \partial_\nu u|_{\partial\Omega}
\]
in the isotropic scalar case, with the usual conormal derivative in the
anisotropic case.  Neumann-to-Dirichlet maps, partial boundary maps, and the
boundary maps for elliptic systems play the same role in other problems.  In
all cases the data are represented by an operator between boundary trace spaces,
and the inverse problem is to recover some part of the interior coefficient from
this operator.  The conductivity problem goes back to Calder\'on
\cite{Calderon1980}; foundational uniqueness results in this direction include
Kohn--Vogelius \cite{KohnVogelius1984,KohnVogelius1985},
Sylvester--Uhlmann \cite{SylvesterUhlmann1987} in dimensions at least three,
and the corresponding planar result of Nachman \cite{Nachman1996}.

There are two distinct questions.  The first is uniqueness: does equality of the
boundary maps force equality of the coefficients, or of the part of the
coefficient one wants to recover?  The second is stability.  A stability estimate
controls the error in the recovered coefficient in terms of the error in the
measured boundary map.  In a typical form one seeks
\[
        \|\gamma_1-\gamma_2\|_{\mathcal C}
        \le
        \omega\bigl(\|\Lambda_{\gamma_1}-\Lambda_{\gamma_2}\|\bigr),
        \qquad \omega(t)\to0\quad(t\to0),
\]
where \(\|\cdot\|_{\mathcal C}\) denotes a chosen norm on the admissible
coefficient class.  The modulus \(\omega\) is important.  Lipschitz and
H\"older estimates are power-type estimates, while logarithmic estimates give
much weaker control.  This distinction is not merely technical: boundary
measurements are never exact, and the modulus describes how errors in the data
propagate into the reconstruction.

In broad infinite-dimensional coefficient classes, elliptic inverse boundary
value problems are severely ill-posed.  For the Calder\'on problem,
Alessandrini proved a logarithmic stability estimate for conductivities
satisfying suitable a priori bounds \cite{Alessandrini1988}.  Mandache's
exponential instability construction \cite{Mandache2001} shows that logarithmic
behavior is essentially unavoidable in such unrestricted settings.  Thus exact
uniqueness does not by itself imply useful quantitative stability in the usual
infinite-dimensional regimes.

The picture changes when the coefficient is known a priori to belong to a
finite-dimensional or otherwise geometrically restricted class.  The prototype
is the Lipschitz stability theorem of Alessandrini--Vessella for piecewise
constant isotropic conductivities on a known partition
\cite{AlessandriniVessella2005}.  Related Lipschitz estimates have been proved
for complex admittivities, Schr\"odinger or Helmholtz potentials, conformal
anisotropic classes, piecewise linear conductivities, and shape-restricted
conductivity inclusions; see, for example,
\cite{BerettaFrancini2011,BerettaDeHoopQiu2013,GaburroSincich2015,AlessandriniDeHoopGaburroSincich2017,AspriBerettaFranciniVessella2022}.
For elasticity, finite-dimensional stability results are available in several
isotropic settings, including piecewise constant Lam\'e parameters, non-flat
interfaces, and time-harmonic elastic waves
\cite{BerettaFranciniVessella2014,BerettaFranciniMorassiRossetVessella2014,BerettaDeHoopFranciniVessellaZhai2017}.
These works typically obtain the stronger Lipschitz rate, but their proofs are
problem-specific and rely on quantitative PDE information such as singular
solutions, propagation of smallness, quantitative unique continuation, or
careful interface arguments.

A related line of work asks whether finite-dimensional a priori structure also
reduces the number of measurements needed.  Alberti and Santacesaria proved
finite-measurement uniqueness, Lipschitz stability, and reconstruction for
finite-dimensional Calder\'on and Schr\"odinger problems
\cite{AlbertiSantacesaria2019}, and later developed a broader framework for
preserving infinite-dimensional inverse-problem information under finite
measurements \cite{AlbertiSantacesaria2022}.  R\"uland--Sincich proved a
finite-Cauchy-data Lipschitz result for a finite-dimensional fractional
Calder\'on problem and later used quantitative Runge approximation in a
finite-dimensional Schr\"odinger problem
\cite{RulandSincich2019,RulandSincich2022}.  Harrach obtained uniqueness and
Lipschitz stability in electrical impedance tomography with finitely many
electrodes for finite-dimensional subsets of piecewise analytic conductivities
\cite{Harrach2019}.  These results are close in spirit to the finite-measurement
part of the present paper, but the mechanism is different: they use additional
PDE, sampling, monotonicity, localized-potential, or reconstruction structure to
obtain Lipschitz estimates.

There are also abstract finite-dimensional stability principles.  Bourgeois
proved a global Lipschitz theorem for nonlinear inverse problems on compact
finite-dimensional sets under \(C^1\) regularity, injectivity of the forward
map, and injectivity of the Fr\'echet derivative on the finite-dimensional
subspace \cite{Bourgeois2013}.  More recently,
Alberti--Arroyo--Santacesaria studied inverse problems on low-dimensional
manifolds, obtaining H\"older and Lipschitz stability results and finite
discretizations of the data under manifold and embedding hypotheses
\cite{AlbertiArroyoSantacesaria2023}.  The present paper belongs to this
finite-dimensional abstract tradition, but uses a different hypothesis package:
real analyticity of the forward map and exact recovery.  No injectivity of the
linearized map is assumed.  The price is that the stability exponent and the
finite measurement set are nonconstructive, and the resulting general estimate
is H\"older rather than Lipschitz.

The purpose of this paper is to isolate this analytic mechanism.  We prove that,
in a finite-dimensional real analytic family, an exact recovery theorem already
forces a H\"older stability estimate on compact admissible parameter sets.  The
source of stability is not a quantitative unique continuation estimate for the
underlying PDE, but the finite-dimensional analytic geometry of the data
mismatch.  This makes the result useful as a transfer principle: once analytic
dependence of the boundary map and an exact uniqueness theorem are known, one
obtains a qualitative global H\"older estimate without reproving a
problem-specific quantitative stability theorem.

Let \(U\subset\R^m\) be open.  We consider a measurement map
\[
        F:U\to \cL(X,Y),
\]
where \(X\) and \(Y\) are separable Hilbert spaces, and a recovered quantity
\[
        R:U\to\R^r .
\]
The case \(R(p)=p\) corresponds to full recovery.  Allowing a general \(R\) is
useful for partial recovery and for layer-stripping theorems, where the data may
determine only the cells reachable from the measured boundary.

The main abstract result is the following.  The precise statement and proof are
given in Theorem~\ref{thm:abstract-holder}.

\begin{theoremA}
Assume that \(F:U\to\cL(X,Y)\) is real analytic in the operator norm and that
\(R:U\to\R^r\) is \(C^1\).  Suppose that
\[
        F(p)=F(q)\quad\Longrightarrow\quad R(p)=R(q),
        \qquad p,q\in U .
\]
Then, for every compact set \(K\Subset U\), there exist \(C>0\) and
\(\theta\in(0,1]\) such that
\[
        |R(p)-R(q)|
        \le C\|F(p)-F(q)\|_{\cL(X,Y)}^\theta,
        \qquad p,q\in K .
\]
\end{theoremA}

The proof is based on a scalar analytic reduction.  Choose Hilbert--Schmidt
operators with dense ranges,
\[
        J_X:H_X\to X,\qquad J_Y:H_Y\to Y,
\]
and set
\[
        \Phi(p,q)
        =
        \|J_Y^*(F(p)-F(q))J_X\|_{\HS}^2 .
\]
Then \(\Phi\) is a scalar real analytic function on \(U\times U\), and the dense
range condition gives
\[
        \Phi(p,q)=0
        \quad\Longleftrightarrow\quad
        F(p)=F(q).
\]
Thus \(R(p)-R(q)\) vanishes on the zero set of \(\Phi\).  The \L{}ojasiewicz
distance inequality controls the distance from \((p,q)\) to this zero set by a
power of \(|\Phi(p,q)|\), and a \(C^1\) estimate for \(R(p)-R(q)\) gives the
H\"older bound.

The same argument also has a finite-measurement form.  Fix countable dense
families \((x_i)_{i\ge1}\subset X\) and \((y_j)_{j\ge1}\subset Y\), and define
\[
        m_{ij}(p)=(F(p)x_i,y_j)_Y .
\]
If the natural data space is \(Y^*\), this pairing is replaced by the
corresponding duality bracket.  The finite-measurement theorem is proved in
Theorem~\ref{thm:finite-measurements}.

\begin{theoremB}
Under the assumptions of Theorem A, for every compact \(K\Subset U\) there are
finitely many pairs \((i_\ell,j_\ell)\), \(\ell=1,\ldots,M\), and constants
\(C>0\), \(\theta\in(0,1]\), such that
\[
        |R(p)-R(q)|
        \le C
        \left(\sum_{\ell=1}^M
        |m_{i_\ell j_\ell}(p)-m_{i_\ell j_\ell}(q)|^2
        \right)^{\theta/2}
\]
for all \(p,q\in K\).
\end{theoremB}

Here the additional input is finite determinacy of real analytic zero sets.  The
countable family of equations
\[
        m_{ij}(p)-m_{ij}(q)=0
\]
has common zero set \(F(p)=F(q)\).  Cartan coherence implies that, locally in
\(U\times U\), this same zero set is defined by finitely many of these
equations.  A compactness argument gives a finite list on a neighborhood of
\(K\times K\).  The finite list depends on \(K\), the chosen dense families, and
the analytic forward map; the theorem does not claim that the measurements are
optimal or effectively computable.

The two inverse problems treated below are meant to illustrate the use of the
abstract theorem, not to delimit its range of applicability.  The first
application concerns piecewise constant anisotropic conductivities.  Let
\[
        \sigma_A(x)=\sum_{j=1}^N A_j\chi_{D_j}(x),
\]
where the \(D_j\) form a known partition and the matrices \(A_j\) are positive
definite and symmetric.  Under the geometric assumptions of
Alessandrini--de Hoop--Gaburro \cite{AlessandriniDeHoopGaburro2017}, equality of
the local Neumann-to-Dirichlet maps on the measured boundary patch implies
\(A_j=B_j\) for all \(j\).  Since the map
\(A\mapsto\cN_A^\Sigma\) is real analytic, Theorem A gives, on every fixed
ellipticity class,
\[
        \max_{1\le j\le N}\|A_j-B_j\|
        \le C
        \|\cN_A^\Sigma-\cN_B^\Sigma\|^\theta .
\]
The finite-measurement theorem gives the corresponding estimate with finitely
many current--voltage pairings.  This should be compared with the Lipschitz
finite-dimensional conductivity results mentioned above, and also with the
boundary stability result for anisotropic conductivities in
\cite{AlessandriniGaburroSincich2024}; the present corollary has a weaker,
nonconstructive H\"older rate but applies by importing the exact local
anisotropic uniqueness theorem of \cite{AlessandriniDeHoopGaburro2017}.

The second application is the localized Dirichlet-to-Neumann map for static
anisotropic elasticity.  We use the layer-stripping uniqueness theorem of
C\^arstea--Honda--Nakamura \cite{CarsteaHondaNakamura2018} for piecewise
homogeneous tensors.  If \(\cR\) denotes the set of cells reachable from the
measured boundary cell through the admissible curved-interface chains, then
\[
        \Lambda_{\mathbb C}^{\Gamma}
        =
        \Lambda_{\widetilde{\mathbb C}}^{\Gamma}
        \quad\Longrightarrow\quad
        \mathbb C^\alpha=\widetilde{\mathbb C}^{\alpha},
        \qquad \alpha\in\cR .
\]
The analytic dependence of
\(\mathbb C\mapsto\Lambda_{\mathbb C}^{\Gamma}\) gives
\[
        \max_{\alpha\in\cR}
        |\mathbb C^\alpha-\widetilde{\mathbb C}^{\alpha}|
        \le C
        \|\Lambda_{\mathbb C}^{\Gamma}
        -\Lambda_{\widetilde{\mathbb C}}^{\Gamma}\|^\theta .
\]
If every cell is reachable, this is a full coefficient estimate.  The closest
previous elasticity stability results concern isotropic Lam\'e parameters or
related dynamic models rather than the localized anisotropic piecewise
homogeneous setting treated here.

The qualitative nature of the result should be kept in mind.  The constants,
exponents, and finite measurement sets come from \L{}ojasiewicz exponents and
local finite generation of analytic ideals associated with the forward map, and
the proof does not provide effective values.  This is why the theorem should be
viewed as a general compact analytic principle rather than as a replacement for
constructive Lipschitz stability.  Its contribution is the following division of
labor: the inverse-problem content is supplied by an exact uniqueness theorem,
while the passage from exact recovery to global H\"older stability and finite
scalar measurements is supplied by real analyticity and finite-dimensionality.

The paper is organized as follows.  Section~\ref{sec:preliminaries} records the
analytic and Hilbert-space tools used in the proof.  Section~\ref{sec:abstract}
proves the abstract H\"older theorem.  Section~\ref{sec:finite-measurements}
proves the finite scalar-measurement version.  Sections~\ref{sec:conductivity}
and \ref{sec:elasticity} apply the theorem to the conductivity and elasticity
problems described above.

\section{Preliminaries}\label{sec:preliminaries}

We record the analytic and Hilbert-space facts used in the abstract proof.

\subsection{Banach-valued real analytic maps}

Let $U\subset\R^m$ be open and let $B$ be a real Banach space.  A map
\[
        F:U\to B
\]
is real analytic if, for every $p_0\in U$, there are continuous symmetric multilinear maps $A_k:(\R^m)^k\to B$ such that
\[
        F(p_0+h)=\sum_{k=0}^\infty A_k(h,\ldots,h)
\]
for $h$ small, with convergence in the norm of $B$.  Equivalently, after complexifying $B$, the map extends locally to a holomorphic map of $m$ complex variables.

We shall use only basic closure properties.  Bounded linear maps preserve real analyticity.  Continuous multilinear maps applied to real analytic maps give real analytic maps.  Finally, if $E$ and $F$ are Banach spaces, then inversion
\[
        T\mapsto T^{-1}
\]
is real analytic on the open subset of $\cL(E,F)$ consisting of bounded isomorphisms $E\to F$.  This last fact follows from the Neumann series expansion around any invertible operator.

\subsection{Hilbert--Schmidt scalarization}

The next construction converts an operator-valued equation into a scalar analytic equation without changing its zero set.

\begin{lemma}[Dense Hilbert--Schmidt operators]\label{lem:HS-probes}
Let $X$ be a separable Hilbert space.  Then there exist a Hilbert space $H$ and a Hilbert--Schmidt operator
\[
        J:H\to X
\]
whose range is dense in $X$.
\end{lemma}

\begin{proof}
If $X$ is infinite-dimensional, choose an orthonormal basis $(e_j)_{j\ge1}$ and take $H=\ell^2$.  Define
\[
        J(a_1,a_2,\ldots)=\sum_{j=1}^{\infty}2^{-j}a_j e_j .
\]
Then
\[
        \|J\|_{\HS}^2=\sum_{j=1}^{\infty}4^{-j}<\infty,
\]
so $J$ is Hilbert--Schmidt.  Its range contains every finite linear combination of the $e_j$, because $\sum_{j=1}^N c_j e_j$ is obtained by taking $a_j=2^j c_j$ for $j\le N$ and $a_j=0$ otherwise.  Hence the range is dense.  The finite-dimensional case is the same argument with a finite sum.
\end{proof}

Let $J_X:H_X\to X$ and $J_Y:H_Y\to Y$ be Hilbert--Schmidt operators.  For $B\in\cL(X,Y)$ define
\[
        Q(B)=J_Y^*BJ_X .
\]
Then $Q(B)\in\HS(H_X,H_Y)$ and
\begin{equation}\label{eq:HS-bound}
        \|J_Y^*BJ_X\|_{\HS(H_X,H_Y)}
        \le \|J_Y\|_{\HS}\,\|B\|_{\cL(X,Y)}\,\|J_X\|_{\HS}.
\end{equation}
Indeed,
\[
        \|BJ_X\|_{\HS}\le \|B\|\,\|J_X\|_{\HS},
        \qquad
        \|J_Y^*(BJ_X)\|_{\HS}
        \le \|J_Y^*\|_{\cL(Y,H_Y)}\|BJ_X\|_{\HS},
\]
and $\|J_Y^*\|_{\cL(Y,H_Y)}\le \|J_Y\|_{\HS}$.

The dense range assumption makes this scalarization faithful.

\begin{lemma}[Faithfulness]\label{lem:faithful}
Let $J_X:H_X\to X$ and $J_Y:H_Y\to Y$ have dense ranges.  If $B\in\cL(X,Y)$ and
\[
        J_Y^*BJ_X=0,
\]
then $B=0$.
\end{lemma}

\begin{proof}
For all $a\in H_X$ and $b\in H_Y$,
\[
        0=(J_Y^*BJ_Xa,b)_{H_Y}=(BJ_Xa,J_Yb)_Y .
\]
Since $J_YH_Y$ is dense in $Y$, this gives $BJ_Xa=0$ for every $a\in H_X$.  Since $J_XH_X$ is dense in $X$ and $B$ is continuous, $B=0$.
\end{proof}

\subsection{The \L{}ojasiewicz distance inequality}

The analytic input is the following distance form of the \L{}ojasiewicz inequality.

\begin{theorem}[\L{}ojasiewicz distance inequality]\label{thm:lojasiewicz}
Let $W\subset\R^d$ be open, let $\phi:W\to\R$ be real analytic, and let
\[
        Z=\{w\in W:\phi(w)=0\}.
\]
Assume $Z\ne\varnothing$.  For every compact set $K\Subset W$ there exist constants $C>0$ and $\beta\in(0,1]$ such that
\[
        \dist(w,Z)\le C|\phi(w)|^\beta,
        \qquad w\in K .
\]
\end{theorem}

For background one may consult Bochnak--Coste--Roy \cite[Sec. 2.6]{BochnakCosteRoy1998} or Feehan \cite{Feehan2020}.  We shall apply the theorem to a function of two parameters, $w=(p,q)$.

We also need the following elementary consequence of $C^1$ regularity.

\begin{lemma}[Vanishing on a closed set]\label{lem:vanishing-distance}
Let $W\subset\R^d$ be open, let $\varnothing\ne Z\subset W$ be closed relative to $W$, and let $H:W\to\R^r$ be $C^1$ with $H=0$ on $Z$.  Then for every compact $K\Subset W$ there exists $L_K>0$ such that
\[
        |H(w)|\le L_K\dist(w,Z),
        \qquad w\in K .
\]
\end{lemma}

\begin{proof}
Choose $\rho>0$ such that the closed $\rho$-neighborhood of $K$ is compactly contained in $W$.  Set
\[
        M=\sup\{|DH(z)|:\dist(z,K)\le \rho\},
        \qquad
        A=\sup_{z\in K}|H(z)| .
\]
If $\dist(w,Z)<\rho/2$, choose $\varepsilon>0$ with $\dist(w,Z)+\varepsilon<\rho$, and choose $z\in Z$ such that
\[
        |w-z|\le \dist(w,Z)+\varepsilon .
\]
Every point of the segment from $w$ to $z$ lies within distance $|w-z|<\rho$ of $w\in K$, hence lies in the $\rho$-neighborhood of $K$.  Since $H(z)=0$, the mean value estimate gives
\[
        |H(w)|\le M|w-z|
        \le M(\dist(w,Z)+\varepsilon).
\]
Letting $\varepsilon\downarrow0$ gives $|H(w)|\le M\dist(w,Z)$.

If $\dist(w,Z)\ge\rho/2$, then
\[
        |H(w)|\le A\le (2A/\rho)\dist(w,Z).
\]
Taking $L_K=\max\{M,2A/\rho\}$ proves the result.
\end{proof}

\section{The abstract H\"older theorem}\label{sec:abstract}

We prove the abstract theorem in the operator-valued form needed later.  The statement is written for maps into \(\cL(X,Y)\).  If the natural boundary map takes values in \(\cL(X,Y^*)\), with \(Y\) Hilbert, we apply the theorem to \(\mathcal R_Y^{-1}F\), where \(\mathcal R_Y:Y\to Y^*\) is the Riesz isomorphism.  Equivalently, the same result holds for \(Y^*\)-valued data after measuring \(F(p)-F(q)\) in the corresponding operator norm.  In the finite-measurement version, the scalar observations are then duality pairings \(\dual{F(p)x}{y}\).

\begin{theorem}[Analytic exact recovery implies H\"older stability]\label{thm:abstract-holder}
Let $U\subset\R^m$ be open, let $K\Subset U$ be compact, and let $X,Y$ be real separable Hilbert spaces.  Let
\[
        F:U\to\cL(X,Y)
\]
be real analytic in the operator norm.  Let
\[
        R:U\to\R^r
\]
be $C^1$.  Assume the exact recovery implication
\begin{equation}\label{eq:exact-recovery}
        F(p)=F(q) \quad\Longrightarrow\quad R(p)=R(q),
        \qquad p,q\in U .
\end{equation}
Then there exist constants $C>0$ and $\theta\in(0,1]$ such that
\begin{equation}\label{eq:abstract-holder}
        |R(p)-R(q)|
        \le C\|F(p)-F(q)\|_{\cL(X,Y)}^\theta,
        \qquad p,q\in K .
\end{equation}
\end{theorem}

\begin{proof}
Choose Hilbert--Schmidt operators $J_X:H_X\to X$ and $J_Y:H_Y\to Y$ with dense ranges, as in Lemma \ref{lem:HS-probes}.  Define
\[
        G(p)=J_Y^*F(p)J_X\in\HS(H_X,H_Y).
\]
The map $B\mapsto J_Y^*BJ_X$ is bounded linear from $\cL(X,Y)$ to $\HS(H_X,H_Y)$ by \eqref{eq:HS-bound}.  Hence $G$ is real analytic.

Set
\[
        \Phi(p,q)=\|G(p)-G(q)\|_{\HS}^2,
        \qquad (p,q)\in U\times U .
\]
Since the Hilbert--Schmidt norm comes from an inner product, $\Phi$ is a scalar real analytic function on $U\times U$.  Let
\[
        Z=\{(p,q)\in U\times U:\Phi(p,q)=0\}.
\]
By Lemma \ref{lem:faithful},
\[
        \Phi(p,q)=0
        \quad\Longleftrightarrow\quad
        F(p)=F(q).
\]
The exact recovery hypothesis therefore implies that
\[
        H(p,q):=R(p)-R(q)
\]
vanishes on $Z$.

The diagonal of $U\times U$ is contained in $Z$, so $Z$ is nonempty.  Apply Theorem \ref{thm:lojasiewicz} to $\Phi$ on the compact set $K\times K$.  There exist $C_1>0$ and $\beta\in(0,1]$ such that
\[
        \dist((p,q),Z)
        \le C_1\Phi(p,q)^\beta,
        \qquad p,q\in K .
\]
By Lemma \ref{lem:vanishing-distance}, applied to $H$ on $K\times K$, there is $C_2>0$ such that
\[
        |R(p)-R(q)|
        \le C_2\dist((p,q),Z),
        \qquad p,q\in K .
\]
Combining these estimates gives
\[
        |R(p)-R(q)|
        \le C_1C_2\|G(p)-G(q)\|_{\HS}^{2\beta}.
\]
Finally, \eqref{eq:HS-bound} gives
\[
        \|G(p)-G(q)\|_{\HS}
        \le \|J_Y\|_{\HS}\|J_X\|_{\HS}
        \|F(p)-F(q)\|_{\cL(X,Y)} .
\]
This proves \eqref{eq:abstract-holder} with exponent $\theta_0=2\beta$.  To put the exponent in the conventional range, let
\[
        M_K=\sup_{p,q\in K}\|F(p)-F(q)\|_{\cL(X,Y)}<\infty .
\]
If $M_K=0$, then $F$ is constant on $K$, and the exact recovery implication makes $R$ constant on $K$, so the estimate is trivial.  Otherwise set $\theta=\min\{\theta_0,1\}$.  For $0\le t\le M_K$ one has
\[
        t^{\theta_0}\le M_K^{\theta_0-\theta}t^\theta,
\]
and the preceding estimate gives \eqref{eq:abstract-holder} with exponent $\theta\in(0,1]$ after increasing the constant.
\end{proof}

\begin{remark}[Scope and limitations]\label{rem:scope-limitations}
The conclusion is a qualitative compact stability statement.  The hypotheses are
that the parameter space is finite-dimensional, the measurement map is real
analytic in the operator norm on an open parameter set, and the full boundary
data exactly determine the chosen recovered quantity.  The estimate is obtained
only on compact subsets of the parameter domain; it need not be uniform up to the
boundary of the admissible set.  The theorem also does not assert a Lipschitz
rate, an effective H\"older exponent, or computable constants.  The finite
measurement result below has the same qualitative character: it proves existence
of a finite separating measurement set on each compact set, but not an optimal or
explicitly computable one.
\end{remark}

\begin{remark}[Full, partial, and gauge-invariant recovery]\label{rem:partial-recovery}
Taking $R(p)=p$ gives full recovery.  Taking $R$ to be a projection gives partial recovery.  This is the natural form for layer-stripping results, where the data may determine only the cells reachable from the measured boundary.  If the inverse problem has a finite-dimensional gauge freedom, the theorem applies to any $C^1$ gauge-invariant recovered quantity $R$, or to a gauge-fixed representative on a chart, provided the exact theorem has the form \eqref{eq:exact-recovery}.
\end{remark}

\begin{remark}[Why analyticity is needed]\label{rem:smooth-counterexample}
The conclusion fails for general $C^\infty$ maps.  Let
\[
        \rho(s)=
        \begin{cases}
        e^{-1/s^2},& s\ne0,\\
        0,& s=0,
        \end{cases}
        \qquad
        F(t)=\int_0^t \rho(s)\,ds .
\]
Then $F\in C^\infty(\R)$.  It is strictly increasing: if $t_1<t_2$, then $\int_{t_1}^{t_2}\rho(s)\,ds>0$, since $\rho\ge0$ and $\rho>0$ except at one point.  Hence $F$ is injective.  However $F$ is flat at the origin.  For $t>0$ one has
\[
        0<F(t)\le t e^{-1/t^2},
\]
so no estimate of the form
\[
        |t|\le C|F(t)|^\alpha,
        \qquad C,\alpha>0,
\]
can hold near $0$.  Real analyticity excludes this infinite-order flatness through the \L{}ojasiewicz inequality.
\end{remark}

\section{Finite scalar measurements}\label{sec:finite-measurements}

The preceding theorem uses the full operator norm of \(F(p)-F(q)\).  We next show that, under the same hypotheses, finitely many scalar matrix elements suffice.

Let $(x_i)_{i\ge1}$ be a countable dense subset of $X$ and let $(y_j)_{j\ge1}$ be a countable dense subset of $Y$.  Define
\[
        m_{ij}(p)=(F(p)x_i,y_j)_Y .
\]
If the measurement map takes values in $\cL(X,Y^*)$, replace the inner product by the duality pairing $\dual{F(p)x_i}{y_j}$.

\subsection{The real analytic zero-set input}

We use the following local finite-determinacy property for countable families of
real analytic equations.  This is the only analytic-geometric input needed in
the finite-measurement argument.

\begin{lemma}[Local finite determinacy of analytic zero sets]\label{lem:finite-zero-germ}
Let $V\subset\R^d$ be open, let $w_0\in V$, and let
$(f_\nu)_{\nu\ge1}$ be a countable family of real analytic functions on $V$.
Then there exist $N_0\in\N$ and a neighborhood $O_{w_0}\Subset V$ of $w_0$
such that
\[
        \{f_1=\cdots=f_{N_0}=0\}\cap O_{w_0}
        =
        \left(\bigcap_{\nu=1}^\infty\{f_\nu=0\}\right)\cap O_{w_0} .
\]
\end{lemma}

\begin{proof}
Let $\mathcal O_{w_0}$ be the local ring of germs of real analytic functions at
$w_0$.  This ring is Noetherian; this is a standard consequence of Cartan's
coherence theorem \cite{Cartan1950}.  See also Bochnak--Coste--Roy
\cite[Sec. 2.3]{BochnakCosteRoy1998} for the real analytic setting.  For
$N\ge1$, set
\[
        I_N=([f_1]_{w_0},\ldots,[f_N]_{w_0})\subset \mathcal O_{w_0}.
\]
The ideals $I_N$ form an increasing chain, hence there is $N_0$ such that
$I_N=I_{N_0}$ for all $N\ge N_0$.  Therefore, for every $\nu>N_0$, the germ
$[f_\nu]_{w_0}$ belongs to the ideal generated by
$[f_1]_{w_0},\ldots,[f_{N_0}]_{w_0}$.  Equivalently, the ideal of germs generated
by the full family $([f_\nu]_{w_0})_{\nu\ge1}$ is generated by
$[f_1]_{w_0},\ldots,[f_{N_0}]_{w_0}$.  Hence the analytic set germ defined by
the full family is the same as the analytic set germ defined by this finite
subfamily.  By equality of germs of sets, after shrinking to a neighborhood
$O_{w_0}\Subset V$, these two zero sets agree in $O_{w_0}$.
\end{proof}

\begin{remark}\label{rem:finite-zero-germ-local}
Only the Noetherianity of the local ring of real analytic germs at $w_0$ is used
in the proof.  The final shrinking step is just the passage from equality of set
germs to equality on a sufficiently small representative neighborhood.
\end{remark}

\subsection{The finite-measurement theorem}

\begin{theorem}[Finite scalar measurements]\label{thm:finite-measurements}
Assume the hypotheses of Theorem \ref{thm:abstract-holder}.  For every compact set $K\Subset U$, there exist finitely many pairs
\[
        (x_{i_1},y_{j_1}),\ldots,(x_{i_M},y_{j_M})
\]
from the chosen dense sets such that the finite measurement map
\[
        \cM(p)=\bigl(m_{i_\ell j_\ell}(p)\bigr)_{\ell=1}^M\in\R^M
\]
satisfies
\begin{equation}\label{eq:finite-holder}
        |R(p)-R(q)|
        \le C|\cM(p)-\cM(q)|^\theta,
        \qquad p,q\in K,
\end{equation}
for some constants $C>0$ and $\theta\in(0,1]$.
\end{theorem}

\begin{proof}
Enumerate the pairs $(i,j)$ by a single index $\nu\ge1$, and write
\[
        h_\nu(p,q)=m_{i_\nu j_\nu}(p)-m_{i_\nu j_\nu}(q),
        \qquad (p,q)\in U\times U .
\]
Because $F$ is analytic, each $h_\nu$ is real analytic.  The common zero set of
all the $h_\nu$ is
\[
        Z=\{(p,q)\in U\times U:F(p)=F(q)\}.
\]
Indeed, if all $h_\nu(p,q)$ vanish, then $(F(p)-F(q))x_i$ is orthogonal to
every $y_j$.  Density of $(y_j)$ gives $(F(p)-F(q))x_i=0$ for every $i$, and
density of $(x_i)$ gives $F(p)=F(q)$.

We claim that finitely many of the $h_\nu$ define the same zero set as the full
family in a neighborhood of $K\times K$.  Fix $w_0\in K\times K$.  If
$w_0\notin Z$, then at least one $h_{\nu_0}$ is nonzero at $w_0$; after
shrinking to a neighborhood $O_{w_0}$, this selected function is nonzero
throughout $O_{w_0}$, and both the selected zero set and the full common zero set
are empty there.  If $w_0\in Z$, Lemma \ref{lem:finite-zero-germ} applied to the
sequence $(h_\nu)_{\nu\ge1}$ gives $N(w_0)$ and a neighborhood $O_{w_0}$ on
which the equations $h_1=\cdots=h_{N(w_0)}=0$ have common zero set
$Z\cap O_{w_0}$.

The neighborhoods $O_{w_0}$ cover the compact set $K\times K$.  Choose a finite
subcover and collect all selected functions.  This gives a finite list
$h_1^*,\ldots,h_M^*$ and an open neighborhood $O\subset U\times U$ of
$K\times K$ such that
\[
        h_1^*(p,q)=\cdots=h_M^*(p,q)=0
        \quad\Longleftrightarrow\quad
        (p,q)\in Z,
        \qquad (p,q)\in O .
\]
Equivalently, for the corresponding finite measurement map $\cM$,
\[
        \cM(p)=\cM(q)
        \quad\Longleftrightarrow\quad
        F(p)=F(q),
        \qquad p,q\in K .
\]
By the exact recovery implication, $H(p,q)=R(p)-R(q)$ vanishes on the zero set
of the selected equations in $O$.

Now define
\[
        \Phi_M(p,q)=|\cM(p)-\cM(q)|^2,
        \qquad (p,q)\in O .
\]
This is a scalar real analytic function.  Its zero set in $O$ is exactly the
common zero set of the selected equations, and hence is contained in the zero
set of $H$.  It is nonempty because it contains the diagonal of $K\times K$.
Applying the \L{}ojasiewicz-distance argument from the proof of Theorem
\ref{thm:abstract-holder}, now to $\Phi_M$ on $K\times K\Subset O$, gives
\eqref{eq:finite-holder}.
\end{proof}

\begin{remark}\label{rem:finite-nonconstructive}
The theorem does not estimate the number $M$ of measurements.  The finite list depends on the compact set $K$, the chosen dense sets, and the analytic forward map; the constants also depend on the recovered quantity $R$.  The theorem does not identify boundary inputs that are experimentally convenient or stable under noise.  It only says that finite-dimensional analyticity and exact recovery force the existence of some finite set of scalar observations with a H\"older inverse on $K$.
\end{remark}

\section{The conductivity inverse problem}\label{sec:conductivity}

We first apply the abstract theorem to the local Neumann-to-Dirichlet map for piecewise constant anisotropic conductivities.

\subsection{Piecewise constant conductivities and local N-D data}

Let $n\ge3$ and let $\Omega\subset\R^n$ be a bounded domain with Lipschitz boundary.  Let
\[
        \overline\Omega=\bigcup_{j=1}^N\overline{D_j}
\]
be a known finite partition by connected, pairwise nonoverlapping open sets.  Let $\Sigma\subset\partial\Omega$ be the open boundary portion on which currents and voltages are measured.  The precise regularity, non-flatness, and chain assumptions on $(\Omega,\Sigma,(D_j))$ will be those appearing in the uniqueness theorem quoted in Subsection \ref{subsec:cond-uniqueness}.

For a tuple of positive definite symmetric matrices
\[
        A=(A_1,\ldots,A_N),
        \qquad A_j\in\Sym(n),
\]
define the piecewise constant anisotropic conductivity
\[
        \sigma_A(x)=\sum_{j=1}^N A_j\chi_{D_j}(x).
\]
The Neumann currents are taken in the localized zero-mean space
\[
        Y={}_0H^{-1/2}(\Sigma),
\]
where, following the notation of \cite{AlessandriniDeHoopGaburro2017}, this is the subspace of ${}_0H^{-1/2}(\partial\Omega)$ consisting of distributions supported in $\overline\Sigma$.  Equivalently, if $\Delta=\partial\Omega\setminus\Sigma$, then $Y$ is the annihilator of $H^{1/2}_{00}(\Delta)$ inside ${}_0H^{-1/2}(\partial\Omega)$.

For $\psi\in Y$, let $u_A^\psi\in H^1(\Omega)/\R$ be the weak solution of the Neumann problem
\[
        \int_\Omega \sigma_A\nabla u_A^\psi\cdot\nabla v\,dx
        =\dual{\psi}{\Tr v}_{\partial\Omega},
        \qquad v\in H^1(\Omega)/\R .
\]
The local Neumann-to-Dirichlet map is the bounded operator
\[
        \cN_A^\Sigma:Y\to Y^*,
\]
defined by
\[
        \dual{\cN_A^\Sigma\psi}{\eta}
        =\int_\Omega \sigma_A\nabla u_A^\psi\cdot\nabla u_A^\eta\,dx,
        \qquad \psi,\eta\in Y .
\]
This is the weak formulation of the local N-D map in \cite[Definition 2.6]{AlessandriniDeHoopGaburro2017}, with the trace of $u_A^\psi$ tested only against currents supported in the measured patch.

\subsection{The ADHG uniqueness theorem}\label{subsec:cond-uniqueness}

The inverse-problem input is the following uniqueness theorem.

\begin{theorem}[Alessandrini--de Hoop--Gaburro]\label{thm:ADHG-blackbox}
Assume that $\Omega$, $\Sigma$, and the known partition $(D_j)_{j=1}^N$ satisfy the hypotheses of \cite[Theorem 2.1]{AlessandriniDeHoopGaburro2017}.  For clarity, the relevant geometric assumptions are the following, in the notation of the cited theorem.  The boundary $\partial\Omega$ and the cell boundaries $\partial D_j$ are Lipschitz; there is a boundary cell, denoted $D_1$, such that $\partial D_1\cap\Sigma$ contains a non-flat $C^{1,\alpha}$ portion; and for each target cell $D_i$ there is a chain
\[
        D_{j_1}=D_1,
        \quad D_{j_K}=D_i,
\]
such that the successive accumulated unions and complements appearing in the chain are Lipschitz domains and each interface $\partial D_{j_k}\cap\partial D_{j_{k-1}}$ contains a non-flat $C^{1,\alpha}$ portion, with the first such portion contained in the measured patch and the later ones lying in the interior.  The theorem below is a black-box restatement of the cited uniqueness result, and all quantitative constants and geometric terminology are meant in the sense of that paper.  Let
\[
        \sigma^{(k)}(x)=\sum_{j=1}^N \sigma_j^{(k)}\chi_{D_j}(x),
        \qquad k=1,2,
\]
where the matrices $\sigma_j^{(k)}\in\Sym(n)$ are positive definite and satisfy common uniform ellipticity bounds.  If
\[
        \cN_{\sigma^{(1)}}^\Sigma=\cN_{\sigma^{(2)}}^\Sigma,
\]
then
\[
        \sigma^{(1)}=\sigma^{(2)}\quad\text{in }\Omega,
\]
equivalently $\sigma_j^{(1)}=\sigma_j^{(2)}$ for $j=1,\ldots,N$.
\end{theorem}

Although Theorem \ref{thm:ADHG-blackbox} is stated with fixed ellipticity bounds, it gives the exact recovery implication on the whole open cone of positive definite tuples: for any two tuples $A,B$ one chooses ellipticity bounds containing both and applies the theorem.

\subsection{Analyticity of the local N-D map}\label{subsec:cond-reduction}

Let
\[
        U_{\rm cond}
        =\{A=(A_1,\ldots,A_N)\in\Sym(n)^N:
        A_j>0\text{ for every }j\}.
\]
For fixed \(0<\lambda<\Lambda<\infty\), set
\[
        K_{\lambda,\Lambda}
        =\{A\in U_{\rm cond}:\lambda I\le A_j\le\Lambda I,
        \ j=1,\ldots,N\}.
\]
Then \(K_{\lambda,\Lambda}\Subset U_{\rm cond}\).  We take
\[
        F(A)=\cN_A^\Sigma,
        \qquad
        R(A)=A .
\]
Theorem~\ref{thm:ADHG-blackbox} gives
\[
        F(A)=F(B)\quad\Longrightarrow\quad A=B,
        \qquad A,B\in U_{\rm cond}.
\]
It remains to verify the analytic dependence of \(F\) on \(A\).

\begin{lemma}[Analyticity of the local N-D map]\label{lem:cond-analytic}
The map
\[
        A\mapsto \cN_A^\Sigma
\]
is real analytic from \(U_{\rm cond}\) into \(\cL(Y,Y^*)\).
\end{lemma}

\begin{proof}
Let
\[
        V=H^1(\Omega)/\R
\]
with the quotient Hilbert norm induced by the gradient.  For \(A\in U_{\rm cond}\),
define
\[
        a_A(u,v)=\int_\Omega \sigma_A\nabla u\cdot\nabla v\,dx,
        \qquad u,v\in V,
\]
and let
\[
        L_A:V\to V^*,
        \qquad
        \dual{L_Au}{v}=a_A(u,v).
\]
The map \(A\mapsto L_A\) is linear.  Indeed,
\[
        \dual{L_Au}{v}
        =\sum_{j=1}^N\sum_{r,s=1}^n
        (A_j)_{rs}\int_{D_j}\partial_su\,\partial_rv\,dx .
\]
For each \(A\in U_{\rm cond}\) there is \(\lambda_A>0\) such that
\(A_j\ge \lambda_A I\) for all \(j\).  Hence
\[
        a_A(u,u)\ge \lambda_A\|\nabla u\|_{L^2(\Omega)}^2,
\]
and \(L_A\) is an isomorphism \(V\to V^*\) by Lax--Milgram.  Since inversion is
analytic on the open set of bounded isomorphisms, \(A\mapsto L_A^{-1}\) is real
analytic as a map into \(\cL(V^*,V)\).

Define the fixed boundary operator
\[
        T:Y\to V^*,
        \qquad
        \dual{T\psi}{v}=\dual{\psi}{\Tr v}_{\partial\Omega}.
\]
This is well defined on the quotient space \(V=H^1(\Omega)/\R\), because \(\psi\in{}_0H^{-1/2}(\partial\Omega)\) has zero mean and therefore annihilates constants.  The map \(T\) is bounded, and the solution of the Neumann problem with current
\(\psi\) is
\[
        u_A^\psi=L_A^{-1}T\psi .
\]
Let \(T^\sharp:V\to Y^*\) be the adjoint operator defined by
\[
        \dual{T^\sharp u}{\eta}=\dual{T\eta}{u}.
\]
For \(\psi,\eta\in Y\),
\[
\begin{aligned}
        \dual{T^\sharp L_A^{-1}T\psi}{\eta}
        &=\dual{T\eta}{L_A^{-1}T\psi}  \\
        &=a_A(u_A^\eta,u_A^\psi)       \\
        &=\dual{\cN_A^\Sigma\psi}{\eta}.
\end{aligned}
\]
Thus
\[
        \cN_A^\Sigma=T^\sharp L_A^{-1}T .
\]
The two outer factors are fixed bounded operators and the middle factor depends
real analytically on \(A\).  This proves the claim.
\end{proof}

\subsection{Stability corollaries}

\begin{corollary}[H\"older stability for local N-D anisotropic conductivity]\label{cor:cond-holder}
Assume the hypotheses of Theorem \ref{thm:ADHG-blackbox}.  Fix any Euclidean norm on the space of symmetric matrices; changing this norm only changes the constants below.  For every $0<\lambda<\Lambda<\infty$ there exist constants $C>0$ and $\theta\in(0,1]$ such that
\begin{equation}\label{eq:cond-holder}
        \max_{1\le j\le N}\|A_j-B_j\|
        \le C
        \|\cN_A^\Sigma-\cN_B^\Sigma\|_{\cL(Y,Y^*)}^\theta
\end{equation}
for all $A,B\in K_{\lambda,\Lambda}$.
\end{corollary}

\begin{proof}
The preceding subsection gives an analytic map $F(A)=\cN_A^\Sigma$ with values
in $\cL(Y,Y^*)$ and the exact implication $F(A)=F(B)\Rightarrow A=B$ on
$U_{\rm cond}$.  Composing the range with the inverse Riesz isomorphism
$Y^*\to Y$ puts the data in the Hilbert-valued form of Theorem
\ref{thm:abstract-holder}; this fixed isomorphism preserves the exact recovery
implication and changes the operator norm only by a fixed equivalence.  Applying
the theorem on $K_{\lambda,\Lambda}$ gives a H\"older estimate in any Euclidean
norm on $\Sym(n)^N$.  Since all norms on this finite-dimensional space are
equivalent, the estimate is equivalent to \eqref{eq:cond-holder}.
\end{proof}

\begin{corollary}[Finite local N-D measurements]\label{cor:cond-finite}
Under the same assumptions, there exist finitely many boundary currents
\[
        \psi_1,\ldots,\psi_M,
        \eta_1,\ldots,\eta_M\in Y
\]
The currents and tests may be chosen from any prescribed countable dense subset
of $Y$.  Moreover, there are constants $C>0$, $\theta\in(0,1]$ such that
\[
        \max_{1\le j\le N}\|A_j-B_j\|
        \le C
        \left(\sum_{\ell=1}^M
        \left|
        \dual{(\cN_A^\Sigma-\cN_B^\Sigma)\psi_\ell}{\eta_\ell}
        \right|^2\right)^{\theta/2}
\]
for all $A,B\in K_{\lambda,\Lambda}$.
\end{corollary}

\begin{proof}
Apply Theorem \ref{thm:finite-measurements} to the analytic map $A\mapsto\cN_A^\Sigma$ and the recovered quantity $R(A)=A$.
\end{proof}

\section{The elasticity inverse problem}\label{sec:elasticity}

We next consider localized Dirichlet-to-Neumann data for static anisotropic elasticity.

\subsection{Piecewise homogeneous elasticity and localized D-N data}\label{subsec:elasticity-data}

Let $\Omega\subset\R^3$ be a bounded Lipschitz domain with a known finite partition
\[
        \overline\Omega=\bigcup_{\alpha=1}^N\overline{D_\alpha}
\]
by open connected subdomains.  Let $\Gamma\subset\partial\Omega$ be the measured boundary patch contained in the boundary of a distinguished boundary cell $D_{\alpha_0}$.

An elasticity tensor $\mathbb C$ is a fourth-order tensor satisfying the usual elasticity symmetries.  In the notation of \cite{CarsteaHondaNakamura2018}, these include
\[
        C_{ijkl}=C_{ijlk},
        \qquad
        C_{ijkl}=C_{klij}.
\]
Together these imply the other minor symmetry, since
\[
        C_{jikl}=C_{klji}=C_{klij}=C_{ijkl}.
\]
The strong convexity condition is that there exists $\lambda>0$ such that
\[
        \mathbb C E:E\ge \lambda |E|^2
        \qquad\text{for every }E\in\Sym(3).
\]
A piecewise homogeneous tensor has the form
\[
        \mathbb C(x)=\sum_{\alpha=1}^N\mathbb C^\alpha\chi_{D_\alpha}(x),
\]
where each $\mathbb C^\alpha$ is a constant strongly convex elasticity tensor.

For a vector field $u:\Omega\to\R^3$, let $Du$ denote the full gradient matrix.  The equation in \cite{CarsteaHondaNakamura2018} is
\[
        L_{\mathbb C}u=\operatorname{div}(\mathbb C::Du)=0
        \quad\text{in }\Omega,
\]
with Dirichlet boundary data on $\partial\Omega$.  The localized displacement space is
\[
        X=H^{1/2}_{\operatorname{co}}(\Gamma;\R^3),
\]
the closure of boundary vector fields supported compactly in $\Gamma$, viewed as a closed subspace of $H^{1/2}(\partial\Omega;\R^3)$ by zero extension.

For $f\in X$, let $u_{\mathbb C}^f\in H^1(\Omega;\R^3)$ be the weak solution with boundary trace $f$.  The localized Dirichlet-to-Neumann map is
\[
        \Lambda_{\mathbb C}^\Gamma:X\to X^*,
\]
defined by
\[
        \dual{\Lambda_{\mathbb C}^\Gamma f}{g}
        =\int_\Omega Du_{\mathbb C}^f:(\mathbb C::Dv)\,dx,
        \qquad f,g\in X,
\]
where $v\in H^1(\Omega;\R^3)$ is any function whose trace is $g$.  This pairing is independent of the chosen extension $v$ by the weak equation.  Equivalently, $\Lambda_{\mathbb C}^\Gamma f$ is the traction $(\mathbb C::Du_{\mathbb C}^f)\nu$ restricted to the measured patch and paired with localized boundary displacements.

For the analytic reduction in Subsection \ref{subsec:elastic-reduction}, it is convenient to use the symmetric gradient
\[
        \varepsilon(u)=\frac{\nabla u+(\nabla u)^T}{2}
\]
and the bilinear form
\[
        a_{\mathbb C}(u,v)=
        \int_\Omega \mathbb C(x)\varepsilon(u):\varepsilon(v)\,dx .
\]
The tensor symmetries imply that this symmetric-gradient formulation agrees with the full-gradient weak form above.  Indeed, if $A,B\in\R^{3\times3}$ and
\[
        \mathbb C A:B=\sum_{i,j,k,l=1}^3 C_{ijkl}A_{ij}B_{kl},
\]
then the minor symmetries give
\[
        \sum_{i,j,k,l}C_{ijkl}A_{ji}B_{kl}
        =\sum_{i,j,k,l}C_{jikl}A_{ij}B_{kl}
        =\sum_{i,j,k,l}C_{ijkl}A_{ij}B_{kl},
\]
and similarly in the second pair of indices.  Hence
\[
        \mathbb C A:B=\mathbb C A^{\rm sym}:B
        =\mathbb C A^{\rm sym}:B^{\rm sym}.
\]
Taking $A=Du$ and $B=Dv$ gives the symmetric-gradient form.  On $H^1_0(\Omega;\R^3)$, strong convexity and Korn's inequality give coercivity; see, for example, Ciarlet \cite{Ciarlet1988}.

\subsection{The CHN chain uniqueness theorem}\label{subsec:elastic-uniqueness}

The exact input is the following layer-stripping theorem.

\begin{theorem}[C\^arstea--Honda--Nakamura]\label{thm:CHN-chain}
Let two piecewise homogeneous anisotropic elasticity tensors be constant on the
same known Lipschitz subdomains $D_\alpha$ and satisfy the symmetry and strong
convexity assumptions above.  Assume the geometric hypotheses of
\cite[Theorem 2.1]{CarsteaHondaNakamura2018}.  In the known-partition form used
here, this means the following.  The measured patch $\Gamma$ lies on a
distinguished boundary cell $D_{\alpha_0}$, and a target cell $D_\alpha$ is
connected to $D_{\alpha_0}$ by an admissible chain
\[
        D_{\alpha_1},D_{\alpha_2},\ldots,D_{\alpha_L},
        \qquad
        \alpha_1=\alpha_0,
        \qquad
        \alpha_L=\alpha,
\]
where successive cells meet along interface patches satisfying the curvedness
condition in the cited theorem.  More precisely, the first patch is contained in the measured
boundary patch, and each later interface patch is an open curved portion of
$\overline D_{\alpha_{\ell-1}}\cap\overline D_{\alpha_\ell}$, with the required
Lipschitz separation assumptions for the accumulated unions and complements.

If
\[
        \Lambda_{\mathbb C^{(1)}}^\Gamma
        =\Lambda_{\mathbb C^{(2)}}^\Gamma,
\]
then
\[
        \mathbb C^{(1),\alpha}
        =\mathbb C^{(2),\alpha}
\]
for every cell $D_\alpha$ reachable from $D_{\alpha_0}$ by such an admissible
chain.  Consequently, if every cell is reachable, then
\[
        \mathbb C^{(1)}=\mathbb C^{(2)}
        \quad\text{in }\Omega .
\]
\end{theorem}

Let $\cR\subset\{1,\ldots,N\}$ denote the set of cells reachable from the measured boundary cell by such curved-interface chains.  The theorem immediately implies the exact partial recovery statement
\[
        \Lambda_{\mathbb C^{(1)}}^\Gamma
        =\Lambda_{\mathbb C^{(2)}}^\Gamma
        \quad\Longrightarrow\quad
        \mathbb C^{(1),\alpha}=\mathbb C^{(2),\alpha}
        \quad\text{for every }\alpha\in\cR .
\]
If all cells are reachable, this is full uniqueness.  The partial form just stated is the part of \cite[Theorem 2.1]{CarsteaHondaNakamura2018} preceding its global conclusion: equality is first obtained on the terminal cell of a prescribed admissible chain, and the global statement follows by applying this to every cell.

\subsection{Analyticity of the localized elasticity D-N map}\label{subsec:elastic-reduction}

Let $U_{\rm el}$ be the open subset of the finite-dimensional vector space of tensor tuples
\[
        \mathbb C=(\mathbb C^1,\ldots,\mathbb C^N)
\]
consisting of tuples satisfying the elasticity symmetries and strong convexity in every cell.  For fixed $0<\lambda<\Lambda<\infty$, define
\[
K_{\lambda,\Lambda}^{\rm el}
=
\left\{\mathbb C\in U_{\rm el}:
\begin{gathered}
\lambda |E|^2\le \mathbb C^\alpha E:E\le\Lambda |E|^2,\\
E\in\Sym(3),\quad \alpha=1,\ldots,N
\end{gathered}
\right\}.
\]
With $\cR$ denoting the set of reachable cells from Theorem \ref{thm:CHN-chain}, set
\[
        F(\mathbb C)=\Lambda_{\mathbb C}^\Gamma,
        \qquad
        R(\mathbb C)=(\mathbb C^\alpha)_{\alpha\in\cR} .
\]
The exact consequence of Theorem \ref{thm:CHN-chain} is
\[
        F(\mathbb C)=F(\widetilde{\mathbb C})
        \quad\Longrightarrow\quad
        R(\mathbb C)=R(\widetilde{\mathbb C}),
        \qquad \mathbb C,\widetilde{\mathbb C}\in U_{\rm el}.
\]
It remains to verify the analytic dependence of the localized elasticity D-N map.

\begin{lemma}[Analyticity of the localized elasticity D-N map]\label{lem:elastic-analytic}
The map
\[
        \mathbb C\mapsto\Lambda_{\mathbb C}^\Gamma
\]
is real analytic from $U_{\rm el}$ into $\cL(X,X^*)$.
\end{lemma}

\begin{proof}
We fix one extension operator and solve only for the zero-boundary correction.

Let
\[
        V_0=H^1_0(\Omega;\R^3).
\]
Choose a bounded right inverse of the trace map
\[
        E:H^{1/2}(\partial\Omega;\R^3)\to H^1(\Omega;\R^3),
\]
and, for $f\in X$, apply it to the zero extension of $f$ from $\Gamma$ to $\partial\Omega$.  We still denote the resulting bounded map $X\to H^1(\Omega;\R^3)$ by $E$.

For the symmetric-gradient form
\[
        a_{\mathbb C}(u,v)=
        \int_\Omega \mathbb C(x)\varepsilon(u):\varepsilon(v)\,dx,
\]
define
\[
        L_{\mathbb C}:V_0\to V_0^*,
        \qquad
        \dual{L_{\mathbb C}w}{v}=a_{\mathbb C}(w,v).
\]
The full-gradient formulation used in the statement of the inverse problem agrees with this symmetric-gradient form by the index calculation in Subsection \ref{subsec:elasticity-data}.  The map $\mathbb C\mapsto L_{\mathbb C}$ is linear: its matrix coefficients are finite linear combinations of fixed operators of the form
\[
        (w,v)\mapsto
        \int_{D_\alpha}\varepsilon_{ij}(w)\varepsilon_{k\ell}(v)\,dx .
\]
Strong convexity and Korn's inequality give coercivity on $V_0$, hence $L_{\mathbb C}$ is a bounded isomorphism for every $\mathbb C\in U_{\rm el}$.  Therefore
\[
        \mathbb C\mapsto L_{\mathbb C}^{-1}
\]
is real analytic into $\cL(V_0^*,V_0)$.

Now define two auxiliary operators.  First,
\[
        P_{\mathbb C}:X\to V_0^*,
        \qquad
        \dual{P_{\mathbb C}f}{v}=a_{\mathbb C}(Ef,v),
        \quad v\in V_0 .
\]
Second,
\[
        Q_{\mathbb C}:X\to X^*,
        \qquad
        \dual{Q_{\mathbb C}f}{g}=a_{\mathbb C}(Ef,Eg).
\]
Both $P_{\mathbb C}$ and $Q_{\mathbb C}$ depend linearly, hence analytically, on $\mathbb C$.  If $f\in X$, the solution with boundary value $f$ has the form
\[
        u_{\mathbb C}^f=Ef+w_{\mathbb C}^f,
        \qquad w_{\mathbb C}^f\in V_0.
\]
The weak equation $a_{\mathbb C}(u_{\mathbb C}^f,v)=0$ for every $v\in V_0$ gives
\[
        L_{\mathbb C}w_{\mathbb C}^f=-P_{\mathbb C}f,
        \qquad
        w_{\mathbb C}^f=-L_{\mathbb C}^{-1}P_{\mathbb C}f.
\]
Thus the solution operator is analytic after the fixed lift $Ef$ is separated off.

It remains to write the boundary traction pairing in terms of these operators.  For $z\in V_0$, define
\[
        P_{\mathbb C}^{\sharp}z\in X^*,
        \qquad
        \dual{P_{\mathbb C}^{\sharp}z}{g}
        =\dual{P_{\mathbb C}g}{z}
        =a_{\mathbb C}(Eg,z).
\]
The map $\mathbb C\mapsto P_{\mathbb C}^{\sharp}$ is again linear.  Since $u_{\mathbb C}^f$ solves the homogeneous equation, the D-N pairing may be computed with the extension $Eg$ of the boundary test $g$:
\[
\begin{aligned}
        \dual{\Lambda_{\mathbb C}^\Gamma f}{g}
        &=a_{\mathbb C}(u_{\mathbb C}^f,Eg) \\
        &=a_{\mathbb C}(Ef,Eg)
          -a_{\mathbb C}(L_{\mathbb C}^{-1}P_{\mathbb C}f,Eg) \\
        &=\dual{Q_{\mathbb C}f}{g}
          -\dual{P_{\mathbb C}^{\sharp}L_{\mathbb C}^{-1}P_{\mathbb C}f}{g}.
\end{aligned}
\]
Here $L_{\mathbb C}^{-1}P_{\mathbb C}f\in V_0$, so $P_{\mathbb C}^{\sharp}$ applies to it.  The last equality uses the symmetry of $a_{\mathbb C}$ to rewrite
\[
        a_{\mathbb C}(L_{\mathbb C}^{-1}P_{\mathbb C}f,Eg)
        =a_{\mathbb C}(Eg,L_{\mathbb C}^{-1}P_{\mathbb C}f).
\]
Therefore, as an operator $X\to X^*$,
\[
        \Lambda_{\mathbb C}^\Gamma
        =Q_{\mathbb C}
        -P_{\mathbb C}^{\sharp}L_{\mathbb C}^{-1}P_{\mathbb C}.
\]
This formula expresses the D-N map using only linear functions of $\mathbb C$, the analytic inverse $L_{\mathbb C}^{-1}$, and bounded bilinear composition of operators.  Hence $\mathbb C\mapsto\Lambda_{\mathbb C}^\Gamma$ is real analytic in $\cL(X,X^*)$.

The formula does not depend on the auxiliary choice of $E$.  If the test extension $Eg$ is changed by an element of $V_0$, the extra term pairs to zero with $u_{\mathbb C}^f$ by the weak equation.  If the lift of $f$ is changed, the zero-boundary correction changes by the opposite zero-boundary function, so the resulting solution $u_{\mathbb C}^f$ is unchanged by uniqueness.
\end{proof}

\subsection{Stability corollaries}

\begin{corollary}[H\"older stability for reachable elasticity tensors]\label{cor:elastic-holder}
Assume the hypotheses of Theorem \ref{thm:CHN-chain}.  Fix any Euclidean norm on the finite-dimensional space of elasticity tensors satisfying the stated symmetries; changing this norm only changes the constants below.  For every $0<\lambda<\Lambda<\infty$ there exist constants $C>0$ and $\theta\in(0,1]$ such that
\begin{equation}\label{eq:elastic-holder}
        \max_{\alpha\in\cR}
        |\mathbb C^\alpha-
        \widetilde{\mathbb C}^{\alpha}|
        \le C
        \|\Lambda_{\mathbb C}^\Gamma-
        \Lambda_{\widetilde{\mathbb C}}^\Gamma\|_{\cL(X,X^*)}^{\theta}
\end{equation}
for all $\mathbb C,\widetilde{\mathbb C}\in K_{\lambda,\Lambda}^{\rm el}$.  If all cells are reachable, the maximum may be taken over all $\alpha=1,\ldots,N$.
\end{corollary}

\begin{proof}
The preceding subsection gives an analytic map
$F(\mathbb C)=\Lambda_{\mathbb C}^\Gamma$ with values in $\cL(X,X^*)$ and the
exact implication
$F(\mathbb C)=F(\widetilde{\mathbb C})\Rightarrow
R(\mathbb C)=R(\widetilde{\mathbb C})$ for
$R(\mathbb C)=(\mathbb C^\alpha)_{\alpha\in\cR}$.  Composing with the inverse
Riesz isomorphism $X^*\to X$ puts the data in the Hilbert-valued form of Theorem
\ref{thm:abstract-holder}; this fixed isomorphism preserves exact recovery and
changes the operator norm only by a fixed equivalence.  Applying the theorem on
$K_{\lambda,\Lambda}^{\rm el}$ gives \eqref{eq:elastic-holder}, up to equivalence
of finite-dimensional norms.
\end{proof}

\begin{corollary}[Finite localized elasticity measurements]\label{cor:elastic-finite}
Under the same hypotheses, there exist finitely many localized boundary
displacements
\[
        f_1,\ldots,f_M,
        g_1,\ldots,g_M\in X
\]
The displacements and tests may be chosen from any prescribed countable dense
subset of $X$.  Moreover, there are constants $C>0$, $\theta\in(0,1]$ such that
\[
        \max_{\alpha\in\cR}
        |\mathbb C^\alpha-
        \widetilde{\mathbb C}^{\alpha}|
        \le C
        \left(\sum_{\ell=1}^M
        \left|
        \dual{(\Lambda_{\mathbb C}^\Gamma-
        \Lambda_{\widetilde{\mathbb C}}^\Gamma)f_\ell}{g_\ell}
        \right|^2\right)^{\theta/2}
\]
for all $\mathbb C,\widetilde{\mathbb C}\in K_{\lambda,\Lambda}^{\rm el}$.
\end{corollary}

\begin{proof}
Apply Theorem \ref{thm:finite-measurements} to the analytic map $\mathbb C\mapsto\Lambda_{\mathbb C}^\Gamma$ and the recovered quantity $R(\mathbb C)=(\mathbb C^\alpha)_{\alpha\in\cR}$.
\end{proof}

The examples in Sections~\ref{sec:conductivity} and \ref{sec:elasticity} show how
known exact recovery theorems can be combined with the abstract analytic
argument.  In both cases the inverse-problem input is contained in the cited
uniqueness theorem, while the stability estimate follows from real analyticity
of the corresponding boundary map.

\subsection*{Acknowledgements}
The author was supported by NSTC grant 113-2115-M-A49-018-MY3.

\bibliographystyle{plain}
\bibliography{refs}

@article{AlbertiArroyoSantacesaria2023,
  author  = {Alberti, Giovanni S. and Arroyo, {\'A}ngel and Santacesaria, Matteo},
  title   = {Inverse problems on low-dimensional manifolds},
  journal = {Nonlinearity},
  volume  = {36},
  number  = {1},
  pages   = {734--808},
  year    = {2023},
  note    = {\doi{10.1088/1361-6544/aca73d}}
}

@article{AlbertiSantacesaria2019,
  author  = {Alberti, Giovanni S. and Santacesaria, Matteo},
  title   = {Calderon's inverse problem with a finite number of measurements},
  journal = {Forum of Mathematics, Sigma},
  volume  = {7},
  pages   = {e35},
  year    = {2019},
  note    = {\doi{10.1017/fms.2019.31}}
}

@article{AlbertiSantacesaria2022,
  author  = {Alberti, Giovanni S. and Santacesaria, Matteo},
  title   = {Infinite-dimensional inverse problems with finite measurements},
  journal = {Archive for Rational Mechanics and Analysis},
  volume  = {243},
  number  = {1},
  pages   = {1--31},
  year    = {2022},
  note    = {\doi{10.1007/s00205-021-01718-4}}
}

@article{Alessandrini1988,
  author  = {Alessandrini, Giovanni},
  title   = {Stable determination of conductivity by boundary measurements},
  journal = {Applicable Analysis},
  volume  = {27},
  number  = {1--3},
  pages   = {153--172},
  year    = {1988},
  note    = {\doi{10.1080/00036818808839730}}
}

@article{AlessandriniDeHoopGaburro2017,
  author  = {Alessandrini, Giovanni and {de Hoop}, Maarten V. and Gaburro, Romina},
  title   = {Uniqueness for the electrostatic inverse boundary value problem with piecewise constant anisotropic conductivities},
  journal = {Inverse Problems},
  volume  = {33},
  number  = {12},
  pages   = {125013},
  year    = {2017},
  note    = {\doi{10.1088/1361-6420/aa982d}}
}

@article{AlessandriniDeHoopGaburroSincich2017,
  author  = {Alessandrini, Giovanni and {de Hoop}, Maarten V. and Gaburro, Romina and Sincich, Eva},
  title   = {Lipschitz stability for the electrostatic inverse boundary value problem with piecewise linear conductivities},
  journal = {Journal de Math{\'e}matiques Pures et Appliqu{\'e}es},
  volume  = {107},
  number  = {5},
  pages   = {638--664},
  year    = {2017},
  note    = {\doi{10.1016/j.matpur.2016.10.001}}
}

@article{AlessandriniGaburroSincich2024,
  author  = {Alessandrini, Giovanni and Gaburro, Romina and Sincich, Eva},
  title   = {Determining an anisotropic conductivity by boundary measurements: Stability at the boundary},
  journal = {Journal of Differential Equations},
  volume  = {382},
  pages   = {115--140},
  year    = {2024},
  note    = {\doi{10.1016/j.jde.2023.11.001}}
}

@article{AlessandriniVessella2005,
  author  = {Alessandrini, Giovanni and Vessella, Sergio},
  title   = {Lipschitz stability for the inverse conductivity problem},
  journal = {Advances in Applied Mathematics},
  volume  = {35},
  number  = {2},
  pages   = {207--241},
  year    = {2005},
  note    = {\doi{10.1016/j.aam.2004.12.002}}
}

@article{BerettaDeHoopFranciniVessellaZhai2017,
  author  = {Beretta, Elena and {de Hoop}, Maarten V. and Francini, Elisa and Vessella, Sergio and Zhai, Jian},
  title   = {Uniqueness and {Lipschitz} stability of an inverse boundary value problem for time-harmonic elastic waves},
  journal = {Inverse Problems},
  volume  = {33},
  number  = {3},
  pages   = {035013},
  year    = {2017},
  note    = {\doi{10.1088/1361-6420/aa5bef}}
}

@article{BerettaDeHoopQiu2013,
  author  = {Beretta, Elena and {de Hoop}, Maarten V. and Qiu, Lingyun},
  title   = {Lipschitz stability of an inverse boundary value problem for a {Schr{\"o}dinger}-type equation},
  journal = {SIAM Journal on Mathematical Analysis},
  volume  = {45},
  number  = {2},
  pages   = {679--699},
  year    = {2013},
  note    = {\doi{10.1137/120869201}}
}

@article{BerettaFrancini2011,
  author  = {Beretta, Elena and Francini, Elisa},
  title   = {Lipschitz stability for the electrical impedance tomography problem: the complex case},
  journal = {Communications in Partial Differential Equations},
  volume  = {36},
  number  = {10},
  pages   = {1723--1749},
  year    = {2011},
  note    = {\doi{10.1080/03605302.2011.552930}}
}

@article{BerettaFranciniMorassiRossetVessella2014,
  author  = {Beretta, Elena and Francini, Elisa and Morassi, Antonino and Rosset, Edi and Vessella, Sergio},
  title   = {Lipschitz continuous dependence of piecewise constant {Lam{\'e}} coefficients from boundary data: the case of non-flat interfaces},
  journal = {Inverse Problems},
  volume  = {30},
  number  = {12},
  pages   = {125005},
  year    = {2014},
  note    = {\doi{10.1088/0266-5611/30/12/125005}}
}

@book{BochnakCosteRoy1998,
  author    = {Bochnak, Jacek and Coste, Michel and Roy, Marie-Fran{\c c}oise},
  title     = {Real Algebraic Geometry},
  series    = {Ergebnisse der Mathematik und ihrer Grenzgebiete},
  volume    = {36},
  publisher = {Springer},
  address   = {Berlin},
  year      = {1998}
}

@article{Bourgeois2013,
  author  = {Bourgeois, Laurent},
  title   = {A remark on {Lipschitz} stability for inverse problems},
  journal = {Comptes Rendus Math{\'e}matique},
  volume  = {351},
  number  = {5--6},
  pages   = {187--190},
  year    = {2013},
  note    = {\doi{10.1016/j.crma.2013.04.004}}
}

@incollection{Calderon1980,
  author    = {Calder{\'o}n, Alberto P.},
  title     = {On an inverse boundary value problem},
  booktitle = {Seminar on Numerical Analysis and its Applications to Continuum Physics},
  pages     = {65--73},
  publisher = {Sociedade Brasileira de Matem{\'a}tica},
  address   = {Rio de Janeiro},
  year      = {1980}
}

@article{CarsteaHondaNakamura2018,
  author  = {C{\^a}rstea, C{\u a}t{\u a}lin I. and Honda, Naofumi and Nakamura, Gen},
  title   = {Uniqueness in the inverse boundary value problem for piecewise homogeneous anisotropic elasticity},
  journal = {SIAM Journal on Mathematical Analysis},
  volume  = {50},
  number  = {3},
  pages   = {3291--3302},
  year    = {2018},
  note    = {\doi{10.1137/17M1125662}}
}

@article{Cartan1950,
  author  = {Cartan, Henri},
  title   = {Id{\'e}aux et modules de fonctions analytiques de variables complexes},
  journal = {Bulletin de la Soci{\'e}t{\'e} Math{\'e}matique de France},
  volume  = {78},
  pages   = {29--64},
  year    = {1950},
  note    = {\doi{10.24033/bsmf.1409}}
}

@book{Ciarlet1988,
  author    = {Ciarlet, Philippe G.},
  title     = {Mathematical Elasticity. Vol. I: Three-Dimensional Elasticity},
  publisher = {North-Holland},
  address   = {Amsterdam},
  year      = {1988}
}

@article{Feehan2020,
  author  = {Feehan, Paul M. N.},
  title   = {Resolution of singularities and geometric proofs of the {\L}ojasiewicz inequalities},
  journal = {Geometry \& Topology},
  volume  = {23},
  number  = {7},
  pages   = {3273--3313},
  year    = {2019},
  note    = {\doi{10.2140/gt.2019.23.3273}}
}

@article{GaburroSincich2015,
  author  = {Gaburro, Romina and Sincich, Eva},
  title   = {Lipschitz stability for the inverse conductivity problem for a conformal class of anisotropic conductivities},
  journal = {Inverse Problems},
  volume  = {31},
  number  = {1},
  pages   = {015008},
  year    = {2015},
  note    = {\doi{10.1088/0266-5611/31/1/015008}}
}

@article{KohnVogelius1984,
  author  = {Kohn, Robert V. and Vogelius, Michael},
  title   = {Determining conductivity by boundary measurements},
  journal = {Communications on Pure and Applied Mathematics},
  volume  = {37},
  number  = {3},
  pages   = {289--298},
  year    = {1984},
  note    = {\doi{10.1002/cpa.3160370302}}
}

@article{KohnVogelius1985,
  author  = {Kohn, Robert V. and Vogelius, Michael},
  title   = {Determining conductivity by boundary measurements. {II}. Interior results},
  journal = {Communications on Pure and Applied Mathematics},
  volume  = {38},
  number  = {5},
  pages   = {643--667},
  year    = {1985},
  note    = {\doi{10.1002/cpa.3160380513}}
}

@article{Mandache2001,
  author  = {Mandache, Niculae},
  title   = {Exponential instability in an inverse problem for the {Schr{\"o}dinger} equation},
  journal = {Inverse Problems},
  volume  = {17},
  number  = {5},
  pages   = {1435--1444},
  year    = {2001},
  note    = {\doi{10.1088/0266-5611/17/5/313}}
}

@article{Nachman1996,
  author  = {Nachman, Adrian I.},
  title   = {Global uniqueness for a two-dimensional inverse boundary value problem},
  journal = {Annals of Mathematics},
  volume  = {143},
  number  = {1},
  pages   = {71--96},
  year    = {1996},
  note    = {\doi{10.2307/2118653}}
}

@article{RulandSincich2019,
  author  = {R{\"u}land, Angkana and Sincich, Eva},
  title   = {Lipschitz stability for the finite dimensional fractional {Calder{\'o}n} problem with finite {Cauchy} data},
  journal = {Inverse Problems and Imaging},
  volume  = {13},
  number  = {5},
  pages   = {1023--1044},
  year    = {2019},
  note    = {\doi{10.3934/ipi.2019046}}
}

@article{SylvesterUhlmann1987,
  author  = {Sylvester, John and Uhlmann, Gunther},
  title   = {A global uniqueness theorem for an inverse boundary value problem},
  journal = {Annals of Mathematics},
  volume  = {125},
  number  = {1},
  pages   = {153--169},
  year    = {1987},
  note    = {\doi{10.2307/1971291}}
}

@article{AspriBerettaFranciniVessella2022,
  author  = {Aspri, Andrea and Beretta, Elena and Francini, Elisa and Vessella, Sergio},
  title   = {Lipschitz Stable Determination of Polyhedral Conductivity Inclusions from Local Boundary Measurements},
  journal = {SIAM Journal on Mathematical Analysis},
  volume  = {54},
  number  = {5},
  pages   = {5182--5222},
  year    = {2022},
  note    = {\doi{10.1137/22M1480550}}
}

@article{BerettaFranciniVessella2014,
  author  = {Beretta, Elena and Francini, Elisa and Vessella, Sergio},
  title   = {Uniqueness and {Lipschitz} stability for the identification of {Lam{\'e}} parameters from boundary measurements},
  journal = {Inverse Problems and Imaging},
  volume  = {8},
  number  = {3},
  pages   = {611--644},
  year    = {2014},
  note    = {\doi{10.3934/ipi.2014.8.611}}
}

@article{Harrach2019,
  author  = {Harrach, Bastian},
  title   = {Uniqueness and {Lipschitz} stability in electrical impedance tomography with finitely many electrodes},
  journal = {Inverse Problems},
  volume  = {35},
  number  = {2},
  pages   = {024005},
  year    = {2019},
  note    = {\doi{10.1088/1361-6420/aaf6fc}}
}

@article{RulandSincich2022,
  author  = {R{\"u}land, Angkana and Sincich, Eva},
  title   = {On {Runge} approximation and {Lipschitz} stability for a finite-dimensional {Schr{\"o}dinger} inverse problem},
  journal = {Applicable Analysis},
  volume  = {101},
  number  = {10},
  pages   = {3655--3666},
  year    = {2022},
  note    = {\doi{10.1080/00036811.2020.1738403}}
}

\end{document}